\documentclass{amsart}
\title[Double-sided torus actions and complex geometry on $SU(3)$]
{Double-sided torus actions and complex geometry on $SU(3)$}

\author{Hiroaki Ishida}
\address{Department of Mathematics, Graduate School of Science, Osaka Metropolitan University}
\email{hiroaki.ishida@omu.ac.jp}
\author{Hisashi Kasuya}
\address{Department of Mathematics, Graduate School of Science, Osaka University, Osaka, Japan.}
\email{kasuya@math.sci.osaka-u.ac.jp}
\thanks{The first author is supported by JSPS KAKENHI Grant Number JP20K03592. The second author is supported by JSPS KAKENHI Grant Number JP19H01787, JP21K03248}
\date{\today}


\usepackage{amssymb}
\usepackage{amsmath}
\usepackage{amscd}
\usepackage{amstext}
\usepackage{amsfonts}
\usepackage{amsrefs}
\usepackage{color}
\usepackage{tikz}
\usepackage{tikz-cd}

\newif\ifdebug

\debugtrue
\newtheorem{prop}{Proposition}[section]
\newtheorem*{prop*}{Proposition \referenza}
\newtheorem{thm}[prop]{Theorem}
\newtheorem*{thm*}{Theorem \referenza}
\newtheorem{cor}[prop]{Corollary}
\newtheorem*{cor*}{Corollary \referenza}
\newtheorem{lemm}[prop]{Lemma}

\theoremstyle{remark}
\newtheorem{rema}[prop]{Remark}

\theoremstyle{definition}

\newtheorem{exam}[prop]{Example}
\newcommand{\C}{\mathbb{C}}
\newcommand{\R}{\mathbb{R}}

\newcommand{\Z}{\mathbb{Z}}

\newcommand{\Ima}{\operatorname{Im}}
\newcommand{\Rea}{\operatorname{Re}}
\newcommand{\diag}{\mathrm{diag}}

\newcommand{\cone}{\mathrm{cone}}

\makeatletter

\@addtoreset{equation}{section}
\makeatother

\newcommand{\relmiddle}[1]{\mathrel{}\middle#1\mathrel{}}

\def\co{\colon\thinspace}

\begin{document} 

\maketitle

\begin{abstract}
We construct explicit complex structures and transverse K\"ahler holomorphic foliations on $SU(3)$. These correspond to variations of real quadratic equations on a complex quadric in $\C^{6}$ as generalizations of left-invariant complex structures on $SU(3)$ and an invariant  K\"ahler structure on the flag variety $SU(3)/T $.
Consequently, we obtain orbifold variants of  the flag variety $SU(3)/T $ as quotients of double-sided torus actions.

\end{abstract}
\section{Introduction}

Let $G$ be a compact connected semi-simple Lie group and $T$ a maximal torus.
The homogeneous space $G/T$ is called flag variety and admits a $G$-invariant K\"ahler structure.
Suppose that $G$ is even dimensional.
Then,  Samelson  \cite{Samelson1953} and Wang \cite{Wang1954}  proved that there exist  left-invariant complex structures on $G$.
Moreover, for any left invariant complex structures on $G$, we may choose a maximal torus $T$ of $G$ so that the principal bundle $T\to G\to G/T$ is holomorphic.
We are interested in constructing variants of these geometric objects.

We consider an action of $T \times T$ on $G$ given by $t\cdot g=t_{1}g t_{2}^{-1}$ for $t=(t_{1},t_{2}) \in T\times T$ and $g\in G$.  
Let $T'$ be a subtorus of $T \times T$. We call the restricted action of $T \times T$ to $T'$ the double-sided torus action on $G$. If the double-sided torus action of $T'$ on $G$ is locally free, then the orbit space $G/T'$ has a structure of an orbifold, and the quotient map $G \to G/T'$ is a Seifert fibering. It is natural to ask whether $G$ has a complex structure, $G/T'$ has a K\"ahler structure and $G \to G/T'$ is a holomorphic Seifert fibering. 
We want to clarify a sufficient condition for it in terms of $T'$. 

This paper focuses on the case when $G = SU(3)$. We study the double-sided torus action on $SU(3)$ and its quotient space from the moment maps perspectives. Let $T$ be the maximal compact torus $\left\{ g= \diag(g_1,g_2, g_3) \relmiddle| g_1,g_2, g_3 \in S^1, g_1g_2g_3=1\right\}$ 
in $SU(3)$, where $\diag(g_1,g_2,g_3)$ denotes the diagonal matrix of size $3$ whose $(j,j)$-entry is $g_j$. 
 For $t = (t_1, t_2) \in (S^1)^2$ and $w = (w_1, w_2) \in \Z^2$, by $t^w$ we mean $t_1^{w_1}t_2^{w_2}$. For $s_1,\dots, s_k$ in $\R^2$, by $\cone (s_1,\dots, s_k)$ we mean the cone spanned by $s_1,\dots, s_k$. That is, $\cone (s_1,\dots, s_k) = \{ \sum_{j=1}^k \lambda_js_j \mid \lambda_j \geq 0\}$. 
The following is the main theorem of this paper: 
\begin{thm} \label{thm:mainthm}
	Let $\rho_L, \rho_R \co (S^1)^2 \to T$ be smooth homomorphisms given by 
	\begin{equation*}
		\begin{split}
			\rho_L (t) &= \diag (t^{w_1^L}, t^{w_2^L}, t^{w_3^L}), \\
			\rho_R (t) &= \diag (t^{w_1^R}, t^{w_2^R}, t^{w_3^R}).
		\end{split}
	\end{equation*}
	Here $w_j^L$, $w_j^R \in \Z^2$. Put 
	\begin{equation*}
		A_j := w_j^L - w_1^R, \quad B_j := -w_j^L +w_3^R, \quad C := -w_1^R +w_3^R.
	\end{equation*}
	Assume that $A_j, B_j$ for $j=1,2,3$ and $C$ satisfy 
	\begin{itemize}
		\item[($\star$)] $C \notin \cone (A_i, A_j), \cone (B_i, B_j)$ for all $i,j \in \{1,2,3\}$ and $C \in \cone (A_i, B_j)$ for all $i,j \in \{1,2,3\}$. 
	\end{itemize}
	Then, the following hold: 
	\begin{enumerate}
		\item There exists a $T \times T$-invariant complex structure on $SU(3)$ such that $(\rho_L, \rho_R)((S^1)^2)$-orbits form a transverse K\"ahler holomorphic foliation. 
		\item There exists a $T \times T/(\rho_L, \rho_R)((S^1)^2)$-invariant K\"ahler orbifold structure on the quotient
		 $SU(3)/(\rho_L, \rho_R)((S^1)^2)$.  
	\end{enumerate}
\end{thm}

$G=SU(3)$ is advantageous. 
Our construction is explicit and may not be generalized to general compact Lie groups.
We know that $SU(3)$ is presented by  explicit real quadric equations on a  complex quadric in $\C^{6}$ (see Section \ref{sec:levelset}).
The principal idea is to vary such real quadric equations corresponding to $(\rho_L,\rho_R)$.
An explicit holomorphic foliation defines a complex structure (Haefliger's trick as in \cite{Haefliger1985}), and a transverse K\"ahler structure comes from the standard K\"ahler structure on $\C^{6}$.
These constructions are  closely related to LVM manifolds \cite{LV1997, Meersseman2000}.
In \cite{MV2004}, it is shown that certain non-K\"ahler manifolds (now known as LVM manifolds) admit structures of principal holomorphic Seifert fibering over toric varieties under some conditions. Theorem \ref{thm:mainthm} is a variant of this result by replacing toric varieties with ``twisted" flag varieties. 

Under the assumption ($\star$) in Theorem \ref{thm:mainthm}, we can check that 
 the action of $(S^1)^2$ on $SU(3)$ given by $(\rho_L, \rho_R)$  on $SU(3)$ is free if and only if $\rho_L$ is trivial and $\rho_R$ is an isomorphism (Proposition \ref{prop:free}).
We obtain essentially new  objects only   if  the action of $(S^1)^2$  is not free.
Meanwhile it is shown in \cite{Eschenburg1992} and \cite{EZ2014} implicitly and \cite{GKZ2020} explicitly that there exists a two-dimensional torus $S_{12}$ such that $S_{12}$ is not contained in $1 \times SU(3)$ and the quotient space $SU(3)/S_{12}$ admits a structure of K\"ahler ``manifold". 
However, it is also shown in \cite[Theorem 3.3]{GKZ2020} that there is no K\"ahler structure on $SU(3)/S_{12}$, which is invariant under the action of two-dimensional torus $T \times T/S_{12}$, not like the flag manifold. 
Thus,  such an example seems to be essentially different from our objects obtained by Theorem \ref{thm:mainthm}.
If $\rho_L$ is non-trivial, then a complex structure in Theorem \ref{thm:mainthm} may not be left-invariant.
Non-left-invariant complex structures on compact Lie groups are studied in \cite{LMN2007} and \cite{IK2020}.

The remainder of this paper is organized as follows. In Section \ref{sec:transverseKaehler}, we develop a critical tool to construct a complex manifold equipped with a transverse K\"ahler structure via moment maps. In Section \ref{sec:levelset}, we  construct an embedding of  $SU(3)$ into a complex quadric and see that the image coincides with a level set of a moment map for some torus action. Section \ref{sec:intersection} studies the intersections of real quadrics in a complex quadric. This is an analogue of intersections of special real quadrics in $\C^n$ developed in \cite{MV2004}, and we use similar arguments. Using results obtained in previous sections, we show Theorem \ref{thm:mainthm}. In Section \ref{sec:freeness}, we study the freeness of the action of $(S^1)^2$ on $SU(3)$ under condition $(\star)$. As a result, the quotient $SU(3)/(\rho_L, \rho_R)((S^1)^2)$ has an orbifold singularity, except when $SU(3)/(\rho_L, \rho_R)((S^1)^2)$ is the ordinal flag manifold. We also provide a nontrivial concrete example of $(\rho_L, \rho_R)$ such that $SU(3)/(\rho_L, \rho_R)((S^1)^2)$ has a K\"ahler orbifold structure. 
In  Section \ref{sec:coho}, we compute cohomologies associated with each structure in Theorem \ref{thm:mainthm}.
\bigskip

\noindent {\bf Acknowledgement.} The authors are grateful to the anonymous referee for pointing out unclear points and providing invaluable comments and useful suggestions for improving the text. The first author is supported by JSPS KAKENHI Grant Number JP20K03592. The second author is supported by JSPS KAKENHI Grant Number JP19H01787.

\section{Holomorphic foliations and commuting Hamiltonians}\label{sec:transverseKaehler}
	This section constructs a transverse K\"ahler structure on a real submanifold in a K\"ahler manifold with a commuting Hamiltonian. We begin by recalling some notions and definitions. For details, we mention the books \cite{Audin2004} for symplectic geometry and \cite{MM2003} for foliations as excellent references. Let $M$ be a smooth manifold of dimension $m$. A foliation atlas of codimension $q$ on $M$ is an atlas $\{ \phi_\alpha = (x_{\alpha,1}, \dots, x_{\alpha, q}, y_{\alpha, 1}, \dots, y_{\alpha, m-q}) \co U_\alpha \to V_\alpha \times W_\alpha\}$ such that 
		\begin{itemize}
			\item $U_\alpha \subset M$, $V_\alpha \subset \R^{q}$ and $W_\alpha \subset \R^{m-q}$ are open subsets, 
			\item the transition functions $\phi_{\alpha\beta} = \phi_\beta\circ \phi_\alpha^{-1}$ are of the form 
			\begin{equation*}
				\phi_{\alpha\beta} (x, y) = (g_{\alpha\beta}(x), h_{\alpha\beta}(x,y)) \in \phi_\beta(U_\alpha \cap U_\beta) \subset V_\beta \times W_\beta
			\end{equation*}
			for $(x,y) \in \phi_\alpha(U_\alpha \cap U_\beta) \subset V_\alpha \times W_\alpha$. That is, the first $q$ coordinate transition functions do not depend on last $m-q$ coordinate functions. 
		\end{itemize}
		Each $\phi_\alpha$ is called a foliation chart. A foliation on $M$ is an equivalence class of a foliation atlas. If $M$ is a complex manifold of complex dimension $m$, $V_\alpha$ and $W_\alpha$ are subsets of $\C^{q}$ and $\C^{m-q}$, respectively, and each $\phi_\alpha$ is a biholomorphic map, then  the foliation on $M$ is said to be holomorphic. 
		
		Let $\{ \phi_\alpha = (x_{\alpha,1}, \dots, x_{\alpha, q}, y_{\alpha, 1}, \dots, y_{\alpha, m-q})\}$ be a foliation atlas of codimension $q$ on a smooth manifold $M$ of dimension $m$.  We can obtain an integrable distribution $\mathcal{D}$ of codimension $q$ on $M$ as follows. Let $p \in M$. The subspace $\mathcal{D}_p$ of the tangent space $T_pM$ at $p$ spanned by 
		\begin{equation*}
			\left(\frac{\partial}{\partial y_{\alpha, 1}}\right)_p, \dots, \left(\frac{\partial}{\partial y_{\alpha,  m-q}}\right)_p
		\end{equation*} does not depend on the choice of $\alpha$ because $\{\phi_\alpha\}$ is a foliation atlas. By definition, the distribution $\mathcal{D} = \{\mathcal{D}_p\}$ is integrable. Conversely, Frobenius theorem yields that any integrable distribution of codimension $q$ determines a unique foliation structure of codimension $q$. If $M$ is a complex manifold and $\mathcal{D}$ is a holomorphic distribution, then $\mathcal{D}$ determines a holomorphic foliation on $M$. We denote the corresponding distribution for a foliation $\mathcal{F}$ on $M$ by $TF=\{T_pF\}$. Let $N$ be a submanifold of dimension $q$ of $M$. We say that a foliation $\mathcal{F}$ is transverse to $N$ if $T_pM = T_pN \oplus T_pF$ for all $p \in N$. A differential $2$-form $\omega$ on a complex manifold $M$ with a complex structure $J$ is  said to be transverse K\"ahler with respect to a holomorphic foliation $\mathcal{F}$ if the following three conditions hold: (1) $\omega$ is closed, (2) $\omega(-,-) = \omega(J-,J-)$ and (3) $\omega(JX,X) \geq 0$ for all $X \in TM$ and $\omega(JX,X) = 0$ if and only if $X \in T\mathcal{F}$. 
			
	Let $(M, \omega, J)$ be a K\"ahler manifold of complex dimension $m$; that is, $M$ is a smooth manifold of dimension $2m$, $\omega$ is a symplectic form (that is, a nondegenerate closed $2$-form) on $M$, and $J$ is an integrable complex structure on $M$ such that $\omega(J-,J-)=\omega(-,-)$ and $\omega(J-,-)$ is positive definite. For a smooth function $f$ on $M$, a vector field defined by $-i_X\omega = df$ is called the Hamiltonian vector field associated with $f$, where $i_X\omega$ is the interior product of $\omega$ and $X$. If $X$ is a Hamiltonian vector field associated with a smooth function, then $L_X\omega=0$, where $L_X\omega$ is the Lie derivative of $\omega$ with respect to $X$. A vector field $X$ is said to be Killing if $L_X\omega (J-,-) =0$. If $X$ is Killing and Hamiltonian, then we have $L_XJ = 0$. 
	
	In the sequel of this section, we assume that $(M,\omega, J)$ is a K\"ahler manifold of complex dimension $m$, $f_1,\dots, f_n, g_1,\dots, g_n$ are smooth functions on $M$, and $X_1,\dots, X_n, Y_1,\dots, Y_n$  
	are the Hamiltonian vector fields associated to $f_1,\dots, f_n, g_1,\dots, g_n$, respectively. We  also assume that $X_1,\dots, X_n, Y_1,\dots, Y_n$ are Killing and commute with each other. Let $\alpha \in \R^{2n}$ be a regular value of the smooth map 
	\begin{equation*}
		\Phi := (f_1,\dots, f_n, g_1,\dots, g_n) \co M \to \R^{2n}
	\end{equation*} and assume that $\Phi ^{-1}(\alpha) \neq \emptyset$. 
	
	First, we construct a holomorphic foliation $\widetilde{\mathcal{F}}$ on a neighborhood of $N := \Phi ^{-1}(\alpha)$ transverse to $N$. For $i=1,\dots, n$, we define vector fields 
	$Z_i := X_i - JY_i$, $W_i := JX_i +Y_i$ on $M$.  Let $p \in N$. We claim that the subspace $\langle Z_1(p), \dots, Z_n(p), W_1(p), \dots, W_n(p)\rangle$ of $T_pM$ spanned by $Z_1(p), \dots, Z_n(p), W_1(p), \dots, W_n(p)$ is transverse to $T_pN$. 
	Let $a_1,\dots, a_n, b_1,\dots, b_n \in \R$. Assume that $\sum_{i=1}^na_iZ_i(p) + b_iW_i(p) \in T_pN$. Since $N$ is a regular level set of $\Phi$, we have $T_pN = \{ v \in T_pM \mid \omega(X_i(p),v) = \omega(Y_i(p),v)=0 \text{ for all $i$}\}$. Thus, 
	\begin{equation*}
		\begin{split}
			0&= \omega(\sum_{i=1}^nb_iX_i(p)-a_iY_i(p), \sum_{i=1}^n a_iZ_i(p) + b_iW_i(p) )\\ 
			&= \omega(\sum_{i=1}^n b_iX_i(p)-a_iY_i(p), J(\sum_{i=1}^n b_iX_i(p)-a_iY_i(p)))
		\end{split}
	\end{equation*}
	because $X_1,\dots, X_n, Y_1,\dots, Y_n$ are commuting and Hamiltonian. 
	Since the bilinear form $\omega (-,J-)$ is negative definite, we have  $\sum_{i=1}^nb_iX_i(p)-a_iY_i(p) =0$. Since $p$ is a regular point, tangent vectors $X_1(p), \dots, X_n(p), Y_1(p),\dots, Y_n(p)$ are linearly independent. Therefore, $a_1 = \dots = a_n = b_1 = \dots = b_n=0$. This shows that 
	\begin{equation*}
		T_pN \cap \langle Z_1(p), \dots, Z_n(p), W_1(p), \dots, W_n(p)\rangle = 0
	\end{equation*}
	and $\dim \langle Z_1(p), \dots, Z_n(p), W_1(p), \dots, W_n(p)\rangle = 2n$. 
	Because of the dimensions, we have 
	\begin{equation*}
		T_pM = T_pN \oplus \langle Z_1(p), \dots, Z_n(p), W_1(p), \dots, W_n(p)\rangle.
	\end{equation*}
	Since $X_i$ and $Y_i$ are Killing and Hamiltonian, we have $L_{X_i}J = L_{Y_i}J = 0$ for all $i$. Since the vector fields $X_1,\dots, X_n,Y_1,\dots, Y_n$ commute, we have that $[X_i, JX_j] = 0$, $[X_i, JY_j]$, $[Y_i, JX_j]$ and $[Y_i,JY_j] = 0$ for all $i,j$. This, together with the vanishing of the Nijenhuis tensor, implies that the vector fields $JX_1,\dots, JX_n, JY_1,\dots, JY_n$ commute. Thus we have that $Z_1,\dots, Z_n, W_1,\dots, W_n$ commute with each other. Therefore $Z_1,\dots, Z_n, W_1, \dots, W_n$ generates a foliation $\widetilde{\mathcal{F}}$ on a neighborhood of $N$ that is transverse to $N$.  Since $JZ_i = W_i$,  the foliation $\widetilde{\mathcal{F}}$ is a holomorphic foliation. 
	
	We construct a complex structure $J_N$ on $N$ using  the foliation $\widetilde{\mathcal{F}}$. Let $p \in N$ and $\phi_\alpha \co U_\alpha \to V_\alpha \times W_\alpha$ be a foliation chart about $p$. By taking sufficiently small $V_\alpha$ and $W_\alpha$, we have a foliation chart $\phi_\alpha$ such that, for any $z_\alpha \in V_\alpha$, there uniquely exists $w_\alpha(z) \in W_\alpha$ such that $\phi_\alpha^{-1}(z_\alpha,w_\alpha(z_\alpha)) \in N$. Let $\{\phi_\alpha \co U_\alpha \to V_\alpha \times W_\alpha\}$ be the set of all such foliated charts. Then, $\{\phi_\alpha\}$ is a foliated atlas on a neighborhood of $N$. Let $\phi_{\alpha1} \co U_\alpha \to V_\alpha$ be the first factor of $\phi_\alpha$. For $p \in N$, we have $\phi_\alpha(p) = (\phi_{\alpha1}(p), w_\alpha(\phi_{\alpha1}(p))$. Thus, we have 
	\begin{equation*}	
		\begin{split}
			 \phi_{\beta1}|_{U_{\beta1}\cap N}\circ (\phi_{\alpha1}|_{U_{\alpha1}\cap N})^{-1} (z_\alpha) &=  \phi_{\beta1}|_{U_{\beta1}\cap N}(\phi_{\alpha}^{-1}(z_\alpha, w_\alpha(z_\alpha))\\
			 &= g_{\alpha\beta}(z_\alpha). 
		\end{split}
	\end{equation*}
	Therefore, $\{\phi_{\alpha1}|_{N\cap U_\alpha} \co N\cap U_\alpha \to V_\alpha\}$ is a holomorphic atlas on $N$. The complex structure $J_N$ on $N$ is described as follows. Let $\overline{i} \co TN \to TM|_N/T\widetilde{\mathcal{F}}|_N$ be the isomorphism induced by the inclusion $i \co TN \to TM|_N$. The projection is denoted by $\pi \co TM|_N \to TM|_N/T\widetilde{\mathcal{F}}|_N$. Then, $J_N$ is given by $J_N = \overline{i}^{-1}\circ\pi \circ J \circ i$ because the complex structure on $N \cap U_\alpha$ comes from $V_\alpha$. 
	
	Since $N$ is a regular level set of the map $\Phi \co M \to \R^{2n}$ and $X_1,\dots, X_n, Y_1,\dots, Y_n$ commute with each other, we have that $X_1,\dots, X_n, Y_1,\dots, Y_n$ define a foliation $\mathcal{F}$ on $N$. We shall see that $\mathcal{F}$ is holomorphic. Since $J_N = \overline{i}^{-1}\circ\pi \circ J \circ i$, we have $J_NX_i(p) \in JX_i(p) + T_p\widetilde{\mathcal{F}}$. On the other hand, $Y_i(p)-JX_i(p) = W_i(p) \in T_p\widetilde{\mathcal{F}}$. Therefore $J_NX_i(p) = Y_i(p)$. By the same argument, we have $J_NY_i(p) = -X_i(p)$. This shows that $\mathcal{F}$ is a holomorphic foliation on $N$. 
	
	Finally, we show that $\omega|_N$ is a transverse K\"ahler form on $N$ with respect to $\mathcal{F}$ by using similar arguments typically used when considering K\"ahler reduction. Let $T_p\mathcal{F}^\perp$ denote the orthogonal complement of $T_p\mathcal{F}$ in $T_pN$ for the Riemannian metric $\omega(J-, -)$. 
	Since 
	\begin{equation*}
		\begin{split}
			T_p\mathcal{F}^\perp & = \{ v \in T_pN \mid \omega (Jv, u) = 0 \text{ for all $u \in T_p\mathcal{F}$}\}\\
			&= \{ v \in T_pM \mid \omega (v, u) = \omega (Jv, u) =0 \text{ for all $u \in T_p\mathcal{F}$}\},
		\end{split}
	\end{equation*}
	we found that $T_p\mathcal{F}^\perp$ is invariant under the complex structure $J$ on $M$. Therefore, $J_N|_{T_p\mathcal{F}^\perp} = J|_{T_p\mathcal{F}^\perp}$. Let $X, Y \in T_pN$. Take $X',Y' \in T_p\mathcal{F}^\perp$ and $X'', Y'' \in T_p\mathcal{F}$ such that $X = X' + X''$ and $Y= Y'+ Y''$. Then,
	\begin{equation*}
			\omega(J_NX, J_NY) = \omega(JX' + J_NX'', JY' + J_NY'')=\omega(JX', JY') = \omega (X',Y')
	\end{equation*}
	because $J_NX'', J_NY'' \in T_p\mathcal{F}$ and $\omega |_N (-,Z)=0$ for any $Z \in T_p\mathcal{F}$. 
	In contrast, 
	\begin{equation*}
			\omega(X, Y) = \omega (X' + X'', Y'+Y'') = \omega (X', Y')
	\end{equation*}
	by the same reason. Thus, $\omega(J_NX, J_NY) = \omega(X, Y)$, showing that $\omega|_N$ satisfies that 
	\begin{equation*}
		\omega|_N(J_N-,J_N-) = \omega|_N(-,-).
	\end{equation*}
	 Since 
	\begin{equation*}
		\omega (J_NX, X) = \omega (JX' + J_NX'', X' + X'') = \omega (JX', X')
	\end{equation*}
	and $\omega (J-, -)$ is positive definite, we  found that $\omega (J_NX, X) \geq 0$, and the equality holds if and only if $X \in T_p\mathcal{F}$. This shows that $\omega|_N$ is a transverse K\"ahler form on $N$ with respect to $\mathcal{F}$.

	We shall state the conclusion in this section as a theorem for later use. 
	\begin{thm}\label{thm:tK}
		The regular level set $N$ has the complex structure $J_N$ such that the $2$-form $\omega|_N$ is transverse K\"ahler form  with respect to the foliation $\mathcal{F}$ generated by commuting Hamiltonian vector fields.  The foliation $\mathcal{F}$ is holomorphic. 
	\end{thm}
	We have the following corollary by definitions of the complex structure on $N$ and $\omega|_N$: 
	\begin{cor}\label{cor:reduction}
		In addition, if $X_1,\dots, X_n, Y_1,\dots, Y_n$ define an action of the torus $T$ on $M$, then  the quotient space $N/T$ is a K\"ahler orbifold. Moreover, the quotient map $N \to N/T$ is a holomorphic Seifert fibering. 
	\end{cor}
	\begin{rema}\label{rema:equivariant}
		If a group $G$ acts on $M$ and the action of $G$ preserves $J$, $\omega$ and $f_1,\dots, f_n, g_1,\dots, g_n$, then $N$, the complex structure $J_N$, the holomorphic foliation $\mathcal{F}$, and the transverse K\"ahler form $\omega|_N$ are invariant under the action of $G$. In particular, the holomorphic Seifert fibering $N \to N/T$ is $G$-equivariant. 
	\end{rema}

	\begin{rema}\label{rem-cstr}
		Let $A \in GL(2n,\R)$. Define $\Phi ' := (f_1',\dots, f_n', g_1',\dots, g_n') = (f_1,\dots, f_n, g_1,\dots, g_n) A$. Then the Hamiltonian vector fields $X_1',\dots, X_n', Y_1',\dots, Y_n'$ of $f_1',\dots, f_n', g_1',\dots, g_n'$ are given by $(X_1',\dots, X_n', Y_1',\dots, Y_n') = ( X_1,\dots, X_n, Y_1,\dots, Y_n)A$. Thus, they are again Killing and commute with each other. Consider $\alpha' := \alpha A$. Then, $\alpha'$ is a regular value of $\Phi'$ and $\Phi'^{-1}(\alpha') = \Phi^{-1}(\alpha) = N$. Replacing vector fields $X_1,\dots, X_n, Y_1,\dots, Y_n$ to $X_1',\dots, X_n', Y_1',\dots, Y_n'$, we obtain another holomorphic foliation transverse to $N$. Therefore, we obtain another complex structure on $N$.  
	Thus,we obtain the family of complex structures on $N$ parametrized by $GL(2n,\R)/GL(n,\C)$.
	By  the construction of the foliation $\mathcal{F}$,  the complex structures on the leaves  of  the holomorphic foliation $\mathcal{F}$ depend on the parameters.  However, the structure of  transverse K\"ahler ${C}^{\infty}$-foliation $\mathcal{F}$ on $N$ is independent.
	In particular, the K\"ahler orbifold structure on $N/T$ is unique.
	\end{rema}

\section{$SU(3)$ as a level set} \label{sec:levelset}
In this section, we embed $SU(3)$ into a K\"ahler manifold as a level set of a smooth map. Let $SL(3,\C)$ denote the complex special linear group of degree $3$. Let $B$  represent the subgroup of upper triangular matrices in $SL(3, \C)$. Let $U$ denote the subgroup of upper triangular matrices whose diagonal entries are $1$. We will construct a K\"ahler form on $SL(3,\C)/U$ explicitly. Let $V = \C^3$ and $\rho \co SL(3, \C) \to GL(V)$ be the natural representation. 
Let $V^*$ be the dual vector space of $V$. Let $e_1$, $e_2$, and $e_3$ be the standard basis vectors of $V$ and $e_1^*$, $e_2^*$, and $e_3^*$ are the dual basis vectors. Let $\rho^* \co SL(3,\C) \to GL(V^*)$ be the dual of $\rho$. For $A \in SL(3,\C)$, the linear map $\rho^*(A) \co V^* \to V^*$ is given by $(\rho^*(A)(f))(v) = f(\rho(A)^{-1}(v))$ for $f \in V^*$ and $v \in V$. The representation matrix of $\rho^*(A) \co V^* \to V^* $ with respect to the basis vectors $e_1^*$, $e_2^*$ and $e_3^*$ is the cofactor matrix ${}^t\! A^{-1}$. Namely, 
\begin{equation*}
	\begin{pmatrix}
		(\rho^*(A))(e_1^*) & (\rho^*(A))(e_2^*) & (\rho^*(A))(e_3^*) 
	\end{pmatrix}
	=
	\begin{pmatrix}
		e_1^* & e_2^* & e_3^*
	\end{pmatrix}
	{}^t\! A^{-1}.
\end{equation*}
\begin{lemm}\label{lemm:SLstab} The subgroup of stabilizers of $SL(3,\C)$ at $(e_1,e_3^*)$ coincides with $U$. 
\end{lemm}
\begin{proof}
	The subgroup of stabilizers of $SL(3,\C)$ at $e_1 \in V$ is of the form
	\begin{equation*}
		\left\{\begin{pmatrix}
		1 & a_{12} & a_{13}\\
		0 & a_{22} & a_{23}\\
		0 & a_{32} & a_{33}
		\end{pmatrix}
		\mid \det \begin{pmatrix}a_{22} & a_{23}\\ a_{32} & a_{33}\end{pmatrix} = 1\right\}.
	\end{equation*}
	Similarly, the subgroup of stabilizers of $SL(3,\C)$ at $e_3^* \in V^*$ is of the form
	\begin{equation*}
		\left\{\begin{pmatrix}
		a_{11} & a_{12} & a_{13}\\
		a_{21} & a_{22} & a_{23}\\
		0 & 0 & 1
		\end{pmatrix}
		\mid \det \begin{pmatrix}a_{11} & a_{12}\\ a_{21} & a_{22}\end{pmatrix} = 1\right\}.
	\end{equation*}
	The intersection of these subgroups coincides with $U$, proving the lemma.
\end{proof}
Thus, the homogeneous space $SL(3, \C)/U$ can be identified with the $SL(3,\C)$-orbit through $(e_1, e_3^*)$. 
\begin{lemm}\label{lemm:SLorb}
	The $SL(3,\C)$-orbit through $(e_1, e_3^*)$ is $\{(v,f) \in V \oplus V^* \mid f(v) =0, v\neq 0, f\neq 0\}$.
\end{lemm}
\begin{proof}
	Let $A \in SL(3,\C)$. Put $v = \rho (A)(e_1)$ and $f = \rho^*(A) (e_3^*)$. Then, $v \neq 0$, $f \neq 0$ and $f(v) = e_3^*(\rho(A)^{-1}\rho (A)(e_1)) = e_3^*(e_1) = 0$. Thus, the $SL(3,\C)$-orbit through $(e_1, e_3^*)$  is contained in $\{ v \oplus f \in V \oplus V^* \mid f(v) =0, v\neq 0, f\neq 0\}$.

	Let $(v,f) \in V \oplus V^*$ be such that $f(v) =0$, $v \neq 0$ and $f\neq 0$. Since $v \neq 0$, there exists $A_1 \in SL(3, \C)$ such that $\rho(A_1)(e_1) = v$. Then, $\rho^*(A_1)^{-1}(f)$ is a linear combination of $e_2^*$ and $e_3^*$ because $f(v) = 0$ and $\rho^*(A_1)^{-1}(f)(e_1) = f(\rho(A_1)(e_1))$. Let $a, b\in \C$ such that $\rho^*(A_1)^{-1}(f) = ae_2^* + be_3^*$. Since $f \neq 0$, either $a$ or $b$ is non-zero. Put $||f|| := \sqrt{|a|^2 + |b|^2}$ and 
	\begin{equation*}
		A_2 := \begin{pmatrix}
			1 & 0 & 0\\
			0 & ||f||^{-2}\overline{b} &  -||f||^{-2}\overline{a} \\
			0 & a & b
		\end{pmatrix}.
	\end{equation*}
	Then, $A_2 \in SL(3, \C)$ and $\rho^*(A_2)^{-1}(e_3^*) = ae_2^*+be_3^*$. Thus $f = \rho^*(A_1A_2^{-1})(e_3^*)$. 
	
	Moreover, $\rho(A_1A_2^{-1})(e_1) = \rho(A_1)(e_1) = v$. We found that 
	\[(v, f) = (\rho(A_1A_2^{-1})(e_1), \rho^*(A_1A_2^{-1})(e_3^*)),\] 
	showing that $(v, f)$ sits in the $SL(3, \C)$-orbit through $(e_1,e_3^*)$. 
\end{proof}

Let $M$ be the quasi-affine variety in $\C^6 = \C^3 \times \C^3$ defined as
\begin{equation*}
	M := \left\{ (z, w) \in \C^3 \times \C^3 \mid z \neq 0, w\neq 0, \sum_{j=1}^3 z_jw_j=0 \right\}
\end{equation*}
where $z = (z_1,z_2, z_3)$ and $w = (w_1,w_2, w_3)$. By Lemmas \ref{lemm:SLstab} and \ref{lemm:SLorb}, $SL(3,\C)/U$ is isomorphic to $M$ as varieties via the isomorphism given by $SL(3,\C)/ U \ni [A] \mapsto (z, w) \in M$, where 
\begin{equation} \label{eq:M}
	\begin{pmatrix}
		z_1\\ z_2\\ z_3\\ w_1\\ w_2\\ w_3 
	\end{pmatrix}
	= 
	\begin{pmatrix}
		A & O\\
		O & {}^tA^{-1}
	\end{pmatrix}
	\begin{pmatrix}
		1 \\ 0 \\ 0 \\ 0\\ 0 \\ 1
	\end{pmatrix}. 
\end{equation}
Since $M$ is quasi-affine, we have a K\"ahler form on $M$, which is obtained by restricting the standard K\"ahler form on $\C^6$. Moreover, $SL(3,\C)/U$ admits a $T \times T$-action introduced below. We  describe the inherited $T \times T$-action on $M$ and its moment map. 

For $g, h \in T$ and $A \in SL(3,\C)$, we define $(g, h) \cdot A := gAh^{-1}$. This is an action of $T \times T$ on $SL(3, \C)$, and it descends to an action on $SL(3,\C)/U$ because $hUh^{-1} = U$ for any $h \in T$. Suppose that $g = \diag (g_1, g_2, g_3)$ and $h = \diag (h_1, h_2, h_3)$. Then,
\begin{equation}\label{eq:weights}
	\begin{split}
	& \begin{pmatrix}
		gAh^{-1} & O \\
		O & {}^t(gAh^{-1})^{-1}
	\end{pmatrix}
	\begin{pmatrix}
		1 \\ 0 \\ 0 \\ 0 \\ 0 \\ 1
	\end{pmatrix}
	= \begin{pmatrix}
		g & O \\
		O & g^{-1}
	\end{pmatrix}
	\begin{pmatrix}
		A & O \\
		O & {}^tA^{-1}	
	\end{pmatrix}
	\begin{pmatrix}
		h_1^{-1} \\ 0 \\ 0 \\ 0 \\ 0 \\ h_3
	\end{pmatrix}\\
	& = 
	\begin{pmatrix}
		\diag (g_1h_1^{-1}, g_2h_1^{-1}, g_3h_1^{-1}) & O \\
		O & \diag (g_1^{-1}h_3, g_2^{-1}h_3, g_3^{-1}h_3) 
	\end{pmatrix}
	\begin{pmatrix}
		A & O \\
		O & {}^tA^{-1}
	\end{pmatrix}
	\begin{pmatrix}
		1 \\ 0 \\ 0 \\ 0 \\ 0 \\ 1
	\end{pmatrix}. 
	\end{split}
\end{equation}
	For a moment, let $z_1,\dots, z_6$ be the standard coordinates of $\C^6$. 
	Let $x_j = \Rea z_j$ and $y_j = \Ima z_j$ be the real coordinates. Let $\omega_{\mathrm{std}}$ be the standard K\"ahler form on $\C^6$. Then $\omega_{\mathrm{std}}$ is represented as $\omega_{\mathrm{std}} = \sum_{j=1}^6 dx_j \wedge dy_j$. For each one-parameter subgroup $\phi \co \R \to (S^1)^6$, there uniquely exists $(a_1,\dots, a_6)\in \R^6$ such that $\phi(t) = (e^{\sqrt{-1}a_1t}, \dots, e^{\sqrt{-1}a_6t})$ for $t \in \R$. By direct computation, a fundamental vector field $X_\phi$ on $\C^6$ corresponding to $\phi$ is represented as 
	\begin{equation*}
		X_\phi = \sum_{j=1}^6 a_j(-y_j\frac{\partial}{\partial x_j} + x_j\frac{\partial}{\partial y_j}). 
	\end{equation*}
	Thus 
	\begin{equation*}
		i_{X_\phi}\omega_{\mathrm{std}} = \sum_{j=1}^6 -a_j(y_jdy_j+x_jdx_j) = d\left(\sum_{j=1}^6-\frac{a_j}{2}|z_j|^2\right). 
	\end{equation*}
	Namely, a function of the form $\sum_{j=1}^6\frac{a_j}{2}|z_j|^2 +c$ for any constant $c \in \R$ is a Hamiltonian function of a fundamental vector field $X_\phi$ and vice versa. 
	
	From now on, let $z_1,z_2,z_3,w_1,w_2,w_3$ be the standard coordinates of $\C^6$. Let $(a_1,a_2,a_3,b_1,b_2,b_3) \in \R^6$. By definition of $M$, for a one-parameter subgroup $\phi \co \R \to (S^1)^6$ given by $\phi(t) = (e^{\sqrt{-1}a_1t},\dots, e^{\sqrt{-1}b_3t})$, $M$ is invariant under the action of $\phi$ if and only if $a_1+b_1= a_2+b_2=a_3+b_3$. Moreover, by direct computation we can see that a homomorphism
	\begin{equation*}
		T \times T \to \{(\alpha_1,\alpha_2,\alpha_3,\beta_1,\beta_2,\beta_3) \in (S^1)^6 \mid \alpha_1\beta_1=\alpha_2\beta_2=\alpha_3\beta_3 \}
	\end{equation*}
	given by 
	\begin{equation}\label{eq:TTS16}
		(g,h) \mapsto (g_1h_1^{-1}, g_2h_1^{-1}, g_3h_1^{-1}, g_1^{-1}h_3, g_2^{-1}h_3, g_3^{-1}h_3)
	\end{equation}
	for $g = \diag (g_1, g_2, g_3)$ and $h = \diag (h_1, h_2, h_3)$ is surjective. This together with \eqref{eq:M} and \eqref{eq:weights} yields that an action of a one-parameter subgroup of $T \times T$ on $M$ is nothing but an action of a one-parameter subgroup $\phi \co \R \to (S^1)^6$ given by $\phi(t) = (e^{\sqrt{-1}a_1t},\dots, e^{\sqrt{-1}b_3t})$ for some $(a_1,a_2,a_3,b_1,b_2,b_3) \in \R^6$ with $a_1+b_1= a_2+b_2=a_3+b_3$. From now on we assume that $a_1+b_1= a_2+b_2=a_3+b_3$. Then the fundamental vector field $X_\phi$ on $\C^6$ corresponding to $\phi$ tangents to $M$. Let $\iota \co M \to \C^6$ be the inclusion. Let $\omega$ be a K\"ahler form on $M$ obtained by restricting the standard K\"ahler form $\omega_{\mathrm{std}}$ on $\C^6$. Namely, $\iota^*\omega_{\mathrm{std}} = \omega$. Let $f_\phi \co \C^6 \to \R$ be a Hamiltonian function of $X_\phi$ on $\C^6$. Then
	\begin{equation*}
		\begin{split}
			i_{X_\phi|_M}\omega&= i_{X_\phi|_M}\iota^*\omega_{\mathrm{std}}\\
			&=\iota^*i_{X_\phi}\omega_{\mathrm{std}}\\
			&=\iota^*(-df_\phi)\\
			&=-d\iota^*f_\phi.
		\end{split}
	\end{equation*}
	Namely, $f_\phi|_M$ is a Hamiltonian function of $X_\phi|_M$. We state this observation as a proposition for later use. 
\begin{prop}\label{prop:Hamiltonian}
	Let $\omega$ be a K\"ahler form on $M$ obtained by restricting the standard K\"ahler form $\omega_{\mathrm{std}}$ on $\C^6$.
	For a fundamental vector field $X$ of $T \times T$-action on $M$, the Hamiltonian function $f$ of $X$ with respect to $\omega$ can be represented as 
	\begin{equation*}
		f(z, w) = \sum_{j=1}^3 (a_j|z_j|^2 + b_j|w_j|^2) + c
	\end{equation*}
	for some real numbers $a_j, b_j$, $j=1,2,3$ with $a_1+ b_1 = a_2 + b_2 = a_3+b_3$ and $c$. The Hamiltonian vector field of such a function is a fundamental vector field of the $T \times T$-action on $M$. 
\end{prop}
\begin{rema}\label{rema:flow}
	Let $f \co M \to \R$ be the function as in Proposition \ref{prop:Hamiltonian}. The flow of the Hamiltonian vector field of $f$ is given by
	\begin{equation*}
		((z,w),t) \mapsto (e^{2a_1t\sqrt{-1}}z_1, \dots, e^{2b_3t\sqrt{-1}}w_3)
	\end{equation*}
	for $(z,w) \in M$ and $t \in \R$. 
\end{rema}
	
	Let $f_1, f_2 \co M \to \R$ be functions defined by 
	\begin{equation*}
		f_1(z,w) := \sum_{j=1}^3 |z_j|^2, \quad f_2(z,w) := \sum_{j=1}^3 |w_j|^2.
	\end{equation*}
	The flow of the Hamiltonian vector field of $f_1$ is given by 
	\begin{equation*}
		((z,w),t) \mapsto (e^{2t\sqrt{-1}}z_1, e^{2t\sqrt{-1}}z_2, e^{2t\sqrt{-1}}z_3, w_1,w_2,w_3)
	\end{equation*}
	for $(z,w) \in M$ and $t \in \R$. By \eqref{eq:TTS16}, this flow is derived from an action of a one-parameter subgroup $\phi_1 \co \R \to T \times T$ on $M$ given by 
	\begin{equation*}
		\phi_1(t) = (\diag(1,1,1), \diag(e^{-2t\sqrt{-1}},e^{2t\sqrt{-1}}, 1)). 
	\end{equation*}
	Also, the flow of the Hamiltonian vector field of $f_2$ is given by 
	\begin{equation*}
		((z,w),t) \mapsto (z_1,z_2, z_3, e^{2t\sqrt{-1}}w_1,e^{2t\sqrt{-1}}w_2,e^{2t\sqrt{-1}t}w_3)
	\end{equation*}
	for $(z,w) \in M$ and $t \in \R$. By \eqref{eq:TTS16}, this flow is derived from an action of a one-parameter subgroup $\phi_2 \co \R \to T \times T$ on $M$ given by 
	\begin{equation*}
		\phi_2(t) = (\diag(1,1,1), \diag(1,e^{-2t\sqrt{-1}},e^{2t\sqrt{-1}})). 
	\end{equation*}
	Therefore, $f_1$ and $f_2$ generate the action of $T \times T$ on $M$ restricted to $\{1\} \times T$.
	\begin{prop}\label{prop:SU(3)}
		We embed $SU(3)$ into $SL(3,\C)/U$ via the map $A \mapsto [A]$ for $A \in SU(3)$. Via the isomorphism between $SL(3,\C)/U$ and $M$ given by \eqref{eq:M}, the preimage of $(1,1)$ by $(f_1,f_2)$ coincides with $SU(3)$ in $SL(3,\C)/U$. 
	\end{prop}
	\begin{proof}
		For any $A \in SU(3)$, the matrices $A$ and ${}^tA^{-1}$ preserve the standard norm of $\C^3$. 
		Thus the inclusion $SU(3) \subset (f_1,f_2)^{-1}(1,1)$ holds. We show that $SU(3)$ is open and closed in $(f_1,f_2)^{-1}(1,1)$ and $(f_1,f_2)^{-1}(1,1)$ is connected.
		Let $\pi \co \C^6 = \C ^3 \times \C^3 \to \C^3$ be the first projection. Then $\pi|_{(f_1,f_2)^{-1}(1,1)} \co (f_1,f_2)^{-1}(1,1) \to S^5$ is a $S^3$-bundle over $S^5$. In particular, $(f_1,f_2)^{-1}(1,1)$ is a connected manifold of dimension $8$. Since the dimension of $SU(3)$ is the same as $(f_1,f_2)^{-1}(1,1)$, $SU(3)$ is open in $(f_1,f_2)^{-1}(1,1)$. It follows from the compactness of $SU(3)$ that $SU(3)$ is closed in $(f_1,f_2)^{-1}(1,1)$. Therefore $SU(3) = (f_1,f_2)^{-1}(1,1)$, proving the proposition. 
	\end{proof}

\section{Intersection of real quadrics in $M$}\label{sec:intersection}
	In this section, we provide sufficient conditions for level sets of $M$ by commuting Hamiltonians to be nonempty, regular, and compact. Let $A_1, A_2, A_3, B_1, B_2, B_3,C \in \R^2$ be such that $A_1 + B_1 = A_2 + B_2 = A_3 +B_3$, and let $\Phi \co M \to \R^2$ be the map defined by
	\begin{equation*}
		\Phi (z,w) = \sum_{j=1}^3 (A_j|z_j|^2 + B_j|w_j|^2). 
	\end{equation*}
	\begin{prop}\label{prop:nonempty}
		If the condition 
		\begin{itemize}
			\item[\bf (N)] There exist $i, j \in \{1,2,3\}$ and $a, b>0$ such that $i\neq j$ and $C = aA_i + bB_j$
		\end{itemize}
		is fulfilled, and then $\Phi^{-1}(C)$ is nonempty. 
	\end{prop}
	\begin{proof} 
		By the assumption, there exists $a, b \in \R_{>0}$ such that $C = aA_i+bB_j$. Put 
		\begin{equation*}
			z_k = \begin{cases}
				\sqrt{a} & \text{if $k=i$}, \\
				0 & \text{otherwise,}
			\end{cases}\quad 
			w_l = \begin{cases}
				\sqrt{b} & \text{if $l=j$}, \\
				0 & \text{otherwise.}
			\end{cases}
		\end{equation*}
		Then $(z_1,z_2,z_3,w_1,w_2,w_3) \in M$ and $\Phi(z,w) = C$. Therefore $\Phi^{-1}(C)$ is nonempty.
	\end{proof}

	\begin{prop}\label{prop:regular}
		If the condition
		\begin{itemize}
			\item[\bf (R)] $A_i$ and $B_j$ are linearly independent if $i\neq j$
		\end{itemize}
		is fulfilled, then $C$ is a regular value of $\Phi$. 
	\end{prop}
	\begin{proof}
		If $\Phi^{-1}(C) = \emptyset$, then $C$ is a regular value of $\Phi$. Suppose that $\Phi ^{-1}(C) \neq \emptyset$. Let $(z,w) \in \Phi^{-1}(C)$. 
		
		Since $(z,w) \in M$, there exists a pair $(i_0, j_0) \in \{1,2,3\}^2$ such that $z_{i_0}, w_{j_0} \neq 0$. We may choose $(i_0, j_0)$ so that $z_{i_0}, w_{j_0} \neq 0$ and $i_0 \neq j_0$ because $\sum_{j=1}^3 z_jw_j=0$. Let $f, g \co M \to \R$ be the first and second components of $\Phi$. Let $X$ and $Y$ be the Hamiltonian vector fields of $f$ and $g$, respectively. By Proposition \ref{prop:Hamiltonian}, $X$ and $Y$ are fundamental vector fields of the action of $T \times T$ on $M$. Since $X$ and $Y$ are Hamiltonian vector fields of $f$ and $g$, we found that $(z,w)$ is a regular point of $\Phi$ if and only if $X_{(z,w)}, Y_{(z,w)} \in T_{(z,w)}M$ are linearly independent. Since $A_{i_0}$ and $B_{j_0}$ are linearly independent and $z_{i_0}, w_{j_0}\neq 0$, we found that the isotropy subgroup at $(z,w)$ of the $\R^2$-action generated by $X$ and $Y$ is discrete (see Remark \ref{rema:flow}). Thus,  $X_{(z,w)}, Y_{(z,w)}$ are linearly independent. Therefore, $(z,w)$ is a regular point of $\Phi$. Since $(z,w)$ is arbitrary, we have that $C$ is a regular value of $\Phi$. 
	\end{proof}

	\begin{prop}\label{prop:compact}
		If the condition 
		\begin{enumerate}
			\item [\bf (C)] 
			\begin{itemize}
				\item $C \notin \cone (A_1, A_2, A_3), \cone (B_1,B_2,B_3)$,
				\item $\cone (A_1, A_2, A_3, B_1,B_2, B_3)$ has an apex, 
				\item $A_j, B_j \neq 0$ for all $j$
			\end{itemize}
		\end{enumerate}
		is fulfilled, then $\Phi^{-1} (C)$ is compact. 
	\end{prop}
	\begin{proof}
		Since $\cone (A_1, A_2, A_3, B_1,B_2, B_3)$ has an apex and $A_i, B_i \neq 0$ for all $i$, there exists a linear function $\alpha \co \R^2 \to \R$ such that $\alpha(A_i)>0$ and $\alpha(B_i)>0$ for all $i$. Applying $\alpha$ to the equation $\Phi(z,w) = C$, we have  $\sum_{i=1}^3\alpha(A_i)|z_i|^2 + \sum_{j=1}^3\alpha(B_j)|w_j|^2 = \alpha(C)$. Since  all coefficients are positive,  $\Phi^{-1}(C)$ is bounded in $\C^3 \times \C^3$. We show that $\Phi^{-1}(C)$ is closed in $\C^3 \times \C^3$. Define the map $\overline{\Phi} \co \C^3 \times \C^3 \to \R^2$  by 
		\begin{equation*}
			\overline{\Phi} (z,w) = \sum_{j=1}^3 (A_j|z_j|^2 + B_j|w_j|^2)
		\end{equation*}
		for $(z,w) \in \C^3 \times \C^3$. Let $(z,w) \in \overline{\Phi}^{-1}(C)$. Since $C \notin \cone (A_1,A_2,A_3)$,  we have $w \neq 0$. Similarly, we have $z \neq 0$. Therefore $\overline{\Phi}^{-1}(C)$ does not intersect with $\{0\} \times \C^3$ and $\C^3 \times \{0\}$. The level set $\Phi^{-1}(C)$ coincides with the intersection of the closed subsets $\overline{\Phi}^{-1}(C)$ and $\{(z,w) \in \C^3 \times \C^3 \mid \sum_{j=1}^3z_jw_j=0\}$. Thus $\Phi^{-1}(C)$ is closed in $\C^3 \times \C^3$.  Since $\Phi^{-1}(C)$ is closed and bounded in the Euclidean space $\C^3 \times \C^3$,  $\Phi^{-1}(C)$ is compact. 
	\end{proof}
	
	\begin{lemm}\label{lemm:star}
		If the condition
		\begin{itemize}
			\item[($\star$)] $C \notin \cone (A_i, A_j), \cone (B_i, B_j)$ for all $i,j \in \{1,2,3\}$ and $C \in \cone (A_i, B_j)$ for all $i,j \in \{1,2,3\}$
		\end{itemize}
		is fulfilled, then {\bf (N)}, {\bf (R)}, and {\bf (C)} are fulfilled. 
	\end{lemm}
	\begin{proof}
		Conditions {\bf (N)} and {\bf (R)} are obvious.
		We can easily check that $A_j, B_j \neq 0$ for all $j$.

		 Suppose that $C \in \cone (A_1,A_2,A_3)$. By Carath\'eodory's theorem,  there exists $i, j \in \{1, 2, 3\}$ such that $C \in \cone (A_i,A_j)$. This contradicts the assumption. Therefore, $C \notin \cone (A_1,A_2,A_3)$. By the same argument, $C \notin \cone (B_1,B_2,B_3)$. 
		
		To show the existence of the apex, we show that there exists $i_0, j_0$ such that 
		\begin{equation*}
			\cone (A_{i_0}, B_{j_0}) = \cone (A_1,A_2,A_3, B_1,B_2,B_3).
		\end{equation*}  
		Since $C \notin \cone (B_1)$, $B_1$ and $C$ are linearly independent. Let $\beta_1,\gamma$ be the dual basis of $B_1,C$. Then,  $A_i = \beta_1(A_i)B_1 + \gamma(A_i)C$.  Since $C \in \cone (A_i, B_1)$,  we have $\beta_1(A_i)<0$ and $\gamma(A_i)>0$ for all $i$.  Take $i_0 \in \{1,2,3\}$  such that $\beta_1(A_{i_0})^{-1}\gamma(A_{i_0}) \geq \beta_1(A_{i})^{-1}\gamma(A_{i})$ for all $i$. Then,  since $B_1 = \beta_1(A_{i_0})^{-1}(A_{i_0}-\gamma(A_{i_0})C)$, we  have 
		\begin{equation*}
			\begin{split}
				A_i &= \beta_1(A_i)B_1 + \gamma(A_i)C\\
				&=\beta_1(A_i)(\beta_1(A_{i_0})^{-1}(A_{i_0}-\gamma(A_{i_0})C) + \gamma(A_i)C\\
				&= \beta_1(A_i)\beta_1(A_{i_0})^{-1} A_{i_0} + (\gamma(A_i) - \beta_1(A_i)\beta_1(A_{i_0})^{-1}\gamma(A_{i_0}))C.
			\end{split}
		\end{equation*}
		Since the coefficients are nonnegative,  $A_i \in \cone (A_{i_0}, C)$. There exists $j_0$ such that $B_j \in \cone (B_{j_0}, C)$ for all $j$ using the same argument as above. Since $C \in \cone (A_{i_0}, B_{j_0})$,  we have $A_1,A_2,A_3, B_1,B_2,B_3 \in \cone (A_{i_0}, B_{i_0})$. Let $\alpha$ be the sum of the dual of $A_{i_0}, B_{j_0}$. Then, $\alpha(A_j), \alpha(B_j)>0$ for all $j$. This shows that $\cone (A_1,A_2,A_3, B_1,B_2,B_3)$ has the apex $0$. Thus, {\bf (C)} is fulfilled, and this completes the proof. 
	\end{proof}

	\begin{thm}\label{thm:SU3}
		If the condition {\bf ($\star$)} is fulfilled, then $\Phi^{-1}(C)$ is $T\times T$-equivariantly diffeomorphic to $SU(3)$. 
	\end{thm}
	\begin{proof}
		Since  $C \in \cone (A_1, B_1)$, there exists $a, b \in \R_{\geq 0}$ such that $C = aA_1 + bB_1$. Since $C \notin \cone (A_1), \cone (B_1)$, we have $a, b \neq 0$. We put $A^0 = aA_1$, $B^0 = bB_1$. Then, $C = A^0 + B^0$. For $t \in [0,1]$, we define $A_i^t = tA_i + (1-t)A^0$ for $i=1,2,3$, $B_j ^t= tB_j + (1-t)B^0$ for $j=1,2,3$. Then, $A_1^t + B_1^t = A_2^t + B_2^t = A_3^t + B_3^t$. We show $(A_1^t, A_2^t, A_3^t, B_1^t, B_2^t, B_3^t, C)$ satisfies the condition {\bf ($\star$)} first. Suppose that $C \in \cone (A_i^t, A_j^t)$ for some $i, j \in \{1,2,3\}$. Then, $C \in \cone (A_1, A_2, A_3)$. By Carath\'eodory's theorem, there  exists $i', j' \in \{1,2,3\}$ such that $C \in \cone (A_{i'}, A_{j'})$. This contradicts to $(\star)$. By the same argument, $C \notin \cone (B_i^t, B_j^t)$. 
		
		Since $C \in \cone (A_i, B_j)$ for $i, j \in \{1,2,3\}$, there exists $r_{ij}^A, r_{ij}^B \in \R{\geq 0}$ such that $C = r_{ij}^AA_i + r_{ij}^BB_j$. Since $C \notin \cone(A_i), \cone(B_j)$, we have that $r_{ij}^A, r_{ij}^B \neq 0$. 
		We have 
		\begin{equation*}
			\begin{split}
				C &= sC + (1-s)C\\
				&= s(r_{ij}^AA_i + r_{ij}^BB_j) + (1-s)(r_{1j}^AA_1 + r_{1j}^BB_j)\\
				&= sr_{ij}^AA_i + (1-s)r_{1j}^AA_1 + (sr_{ij}^B+(1-s)r_{1j}^B)B_j
			\end{split}
		\end{equation*}
		for all $s \in [0,1]$. Choose $s$ so that $sr_{ij}^AA_i + (1-s)r_{1j}^AA_1 \in \cone (A_i^t)$. There  exist $r_{ij}^A{}', r_{ij}^B{}'>0$ such that $C = r_{ij}^A{}'A_i^t + r_{ij}^B{}'B_j$. By the same argument as above, there exist $r_{ij}^A{}'', r_{ij}^B{}''>0$ such that $C = r_{ij}^A{}''A_i^t + r_{ij}^B{}''B_j^t$. Therefore, $(\star)$ is fulfilled. 
	
		Define the map  $\Psi \co M \times [0,1] \to \R^2$  by \begin{equation*}\Psi(z,w,t) = \sum_{j=1}^3 (A_j^t|z_j|^2 + B_j^t|w_j|^2)\end{equation*}
		 for $(z,w,t) \in M \times [0,1]$. 
		By proposition \ref{prop:regular} and Lemma \ref{lemm:star}, $\Psi^{-1}(C)$ is nonempty compact set, and $C$ is a regular value of $\Psi$. Thus $\Psi ^{-1}(C)$ is a manifold with  boundary. Let $\pi \co \Psi^{-1}(C) \to [0,1]$ be the second projection. Since $(A_1^0, A_2^0, A_3^0, B_1^0, B_2^0, B_3^0, C)$ satisfies {\bf ($\star$)}, it follows from Proposition \ref{prop:regular} and Lemma \ref{lemm:star} that  $A_1^0$ and $B_1^0$ form a basis of $\R^2$. By applying the inverse matrix, it follows from Proposition \ref{prop:SU(3)} that $\pi^{-1}(0)$ is equivariantly diffeomorphic to $SU(3)$. By definition, $\pi$ is an equivariant proper surjective submersion. We apply the following lemma. 
		
		\begin{lemm}\label{lemm:equivariantEhresmann}
			Let $N$ be a compact manifold with boundary, $G$ a compact Lie group acting on $N$, and $\pi \co N \to [0,1]$ a $G$-invariant surjective submersion.  Then the pre-image $\pi^{-1}(1)$ is $G$-equivariantly diffeomorphic to $\pi^{-1}(0)$. 
		\end{lemm}
		\begin{proof}[Proof of Lemma \ref{lemm:equivariantEhresmann}]
			Since $\pi \co N \to [0,1]$ is a submersion, by constant rank theorem, there exists an atlas $\{(U_\alpha, x_{\alpha,1},\dots, x_{\alpha,n})\}$ on $N$ such that $\pi$ is of the form $(x_{\alpha,1},\dots, x_{\alpha,n}) \mapsto x_{\alpha,n}$. Let $\{\rho_\alpha\}$ be a partition of unity subordinate to the open cover $\{U_\alpha\}$. Let $V$ be a vector field on $N$ defined by $V=\sum_{\alpha}\rho_\alpha \partial /\partial x_{\alpha,n}$. Then, $\pi_{*}V = \partial/\partial t$, where $t$ is the standard coordinate of $[0,1]$. 
			We define a vector field $V'$ on $N$ by 
			\begin{equation*}
				V'_p = \int_{g \in G} (g_{*,p})^{-1}(V_{gp}) d\mu
			\end{equation*}
			for $p \in N$, where $\mu$ denotes the normalized Haar measure on $G$. Then $V'$ is a $G$-invariant smooth vector field on $N$. Indeed, for $p \in N$ and $h\in G$,  
			\begin{equation*}
				\begin{split}
					h_{*,p}(V'_p) &= h_{*,p}\int_{g \in G} (g_{*,p})^{-1}(V_{gp}) d\mu\\
					&=\int_{g \in G} h_{*,p}\circ (g_{*,p})^{-1}(V_{gp}) d\mu\\
					&= \int_{g \in G} h_{*,p}\circ (g_{*,gp}^{-1})(V_{gp})d\mu\\
					&= \int_{g \in G} (hg^{-1})_{*,gp}(V_{gp})d\mu\\
					&= \int_{g \in G} (gh^{-1})_{*,hp}^{-1}(V_{gp})d\mu\\
					&= \int_{g \in G} (gh^{-1})_{*,hp}^{-1}(V_{(gh^{-1})hp})d\mu\\
					&= V'_{hp}. 
				\end{split}
			\end{equation*}
			Moreover, since $\pi$ is $G$-invariant and $\pi_{*}V = \partial/\partial t$, we have $\pi_*V' = \partial/\partial t$. Indeed, for $p \in N$, 
			\begin{equation}\label{eq:piV'}
				\begin{split}
					\pi_{*,p}(V_p') &= \pi_{*,p}\int_{g \in G} (g_{*,p})^{-1}(V_{gp}) d\mu\\
					&= \int_{g \in G} \pi_{*,p}\circ (g_{*,p})^{-1}(V_{gp}) d\mu\\
					&= \int_{g \in G} \pi_{*,gp}(V_{gp}) d\mu\\
					&= \int_{g \in G} (\partial/\partial t)_{\pi(gp)} d\mu\\
					&= \int_{g \in G} (\partial/\partial t)_{\pi(p)} d\mu\\
					&= (\partial/\partial t)_{\pi(p)}.
				\end{split}
			\end{equation}
			The third and fifth equalities in \eqref{eq:piV'} follow from $G$-invariance of $\pi$. The last equality in \eqref{eq:piV'} follows from that $\mu$ is the normalized Haar measure on $G$. 
			Since $N$ is compact, we have a flow $F(x,t)$ of $V'$, which is defined whenever $\pi(x)+t \in [0,1]$ for $x \in N$ and $t \in [0,1]$. Since $V'$ is $G$-invariant, so is $F(x,t)$. Therefore, the map given by $x \mapsto F(x,1)$ for $x \in \pi^{-1}(0)$ is a $G$-invariant diffeomorphism onto $\pi^{-1}(1)$. The lemma is proved.
		\end{proof}
		
		Applying Lemma \ref{lemm:equivariantEhresmann}, 
		 $\pi^{-1}(0)= SU(3)$ and $\pi^{-1}(1) = \Phi^{-1}(C)$ are equivariantly diffeomorphic, thus proving the theorem.
	\end{proof}
	Now we can prove Theorem \ref{thm:mainthm}.
\begin{proof}[Proof of Theorem \ref{thm:mainthm}]
	Part (1) follows from Theorems \ref{thm:tK}, \ref{thm:SU3} and Remark \ref{rema:equivariant}. Part (2) follows from Corollary \ref{cor:reduction}, Theorem \ref{thm:SU3} and Remark \ref{rema:equivariant}. 
\end{proof}

\begin{rema}\label{rem-cstr2}
As Remark \ref{rem-cstr}, for each $(\rho_L, \rho_R)$ under the condition $(\star)$, we obtain the family of complex structures on $SU(3)$ parametrized by $GL(2,\R)/GL(1,\C)$.
The  K\"ahler orbifold structure on $SU(3)/(\rho_L, \rho_R)((S^1)^2)$ is uniquely determined by $(\rho_L, \rho_R)$.
\end{rema}

\begin{rema}
	Theorem \ref{thm:mainthm} (1) can be generalized to double-sided $\R^2$-actions on $SU(3)$. Let 
	\begin{equation*}
		 w_1^L, w_2^L, w_3^L, w_1^R, w_2^R, w_3^R \in \R^2
	\end{equation*}
	 be such that $\sum_{j=1}^3 w_j^L = \sum_{j=1}^3w_j^R=0$, and $(\star)$ is fulfilled. 
	  We consider the action of $\R^2$ on $SU(3)$ defined by 
	\begin{equation*}
		v\cdot A := \diag(e^{\sqrt{-1}\langle v, w_1^L\rangle}, e^{\sqrt{-1}\langle v, w_2^L\rangle}, e^{\sqrt{-1}\langle v, w_3^L\rangle})A  \diag(e^{\sqrt{-1}\langle v, w_1^R\rangle}, e^{\sqrt{-1}\langle v, w_2^R\rangle}, e^{\sqrt{-1}\langle v, w_3^R\rangle})^{-1}
	\end{equation*}
	for $v \in \R^2$ and $A \in SU(3)$, where $\langle -, -\rangle$ denotes the standard inner product on $\R^2$.  Then, there exists a $T \times T$-invariant complex structure on $SU(3)$ such that the foliation whose leaves are $\R^2$-orbits is holomorphic and transverse K\"ahler. All leaves are closed if and only if $A_j, B_j$, $j=1,2,3$ generate a lattice of $\R^2$. 

Consider $A^{t}_j , B_j^t\in \R^{2}$  as in the Proof of Theorem \ref{thm:SU3}.
Corresponding to 
 \[w_1^L(t), w_2^L(t), w_3^L(t), w_1^R(t), w_2^R(t), w_3^R(t) \in \R^2\]
  so that 
\begin{equation*}
		A^{t}_j := w_j^L(t) - w_1^R(t), \quad B_j^t := -w_j^L(t) +w_3^R(t), \quad C := -w_1^R(t) +w_3^R(t),
	\end{equation*}
by the above arguments,  we also obtain  a smooth  family $\{J_{t}\}$ of $T \times T$-invariant complex structures on $SU(3)$ parametrized by $t\in [0,1]$.
$J_{0}$ is a left-invariant complex structure on $SU(3)$ (see Remark \ref{remafree}).
In \cite{IK2020}, we compute  the Kuranishi spaces  of left-invariant complex structures on even-dimensional simply connected compact Lie groups.
This allows us to describe all small deformations of  left-invariant complex structures.
However, we use the Kodaira-Spencer theory in \cite{IK2020}. Therefore,  it is difficult to construct large deformations such as $\{J_t\}$, in this manner. 
\end{rema}

\section{Freeness of the double-sided action}\label{sec:freeness}
In this section, we consider the case when $(\rho_L, \rho_R)((S^1)^2)$-action is free under the condition $(\star)$. 
\begin{lemm}\label{lemm:free}
		Let $\rho_L, \rho_R \co (S^1)^2 \to T$ be smooth homomorphisms and $A_j$, $B_j$ as Theorem \ref{thm:mainthm}.  The $(S^1)^2$-action on $SU(3)$ given by $\rho_L, \rho_R$ is free if and only if $A_i, B_j$ form a $\Z$-basis of $\Z^2$ for all pair $(i,j)$ with $i \neq j$. 
	\end{lemm}
	\begin{proof}
		By Proposition \ref{prop:SU(3)}, $SU(3)$ is equivariantly diffeomorphic to \begin{equation*}X_0 := \left\{(z, w) \in \C^3 \times \C^3 \mid \sum_{j=1}^3|z_j|^2=1, \sum_{j=1}^3|w_j|^2=1, \sum_{j=1}^3 z_jw_j=0\right\} \subset M.\end{equation*}
		We shall see the isotropy subgroups of the $T$-action on $X_0$  given by $(\rho_L, \rho_R)$ are trivial. Let $(z, w) \in X_0$. Put $I_z := \{ i \mid z_i \neq 0\}$ and $J_w := \{j \mid w_j \neq 0\}$. 
		Since the $(S^1)^2$-action on $X_0$ given by $(\rho_L, \rho_R)$ satisfies
		\begin{equation*}
			t\cdot (z,w) = (t^{A_1}z_1, t^{A_2}z_2, t^{A_3}z_3, t^{B_1}w_1, t^{B_2}w_2, t^{B_3}w_3)
		\end{equation*}
		by \eqref{eq:weights}, 
		the isotropy subgroup  at $(z,w)$ is isomorphic to 
		\begin{equation*}
			\bigcap_{i \in I_z} \ker (t \mapsto t^{A_i}) \cap \bigcap_{j \in J_w} \ker (t \mapsto t^{B_i}). 
		\end{equation*}
		
		For the ``if" part, assume that $A_i, B_j$ form a $\Z$-basis of $\Z^2$ for all pairs $(i,j)$ with $i \neq j$. Since $z \neq 0$ and $w \neq 0$, we have  $I_z \times J_w \neq \emptyset$. Since $\sum_{j=1}^3 z_jw_j =0$, there exists a pair $(i,j) \in I_z \times J_w$ with $i\neq j$. Since $A_i$ and $B_j$ form a basis of $\Z^2$, the homomorphism $(S^1)^2 \to (S^1)^2$ given by $t \to (t^{A_i}, t^{B_j})$ is an isomorphism. In particular, $\ker (t \mapsto t^{A_i}) \cap \ker (t \mapsto t^{B_i})$ is trivial. 
		Thus, the isotropy subgroup at $(z,w)$ is trivial. Since $(z, w) \in X_0$ is arbitrary,  the $(S^1)^2$-action on $X_0$ given by $(\rho_L, \rho_R)$ is free. 
		
		For the ``only if" part, assume that the $(S^1)^2$-action  on $X_0$ given by $(\rho_L, \rho_R)$ is free. Let $i_0, j_0 \in \{1,2,3\}$ be such that $i_0 \neq j_0$. Put 
		\begin{equation*}
			z_i =\begin{cases} 
				0, & i \neq i_0,\\
				1, & i =i_0,
			\end{cases}
			\quad 
			w_j = \begin{cases}
				0, & j \neq j_0,\\
				1, & j =j_0.
			\end{cases}
		\end{equation*}
		Then $(z,w) := (z_1, z_2, z_3, w_1, w_2, w_3) \in X_0$. Since the isotropy subgroup at $(z,w)$ is trivial, $A_i$ and $B_j$ span $\Z^2$. Therefore, $A_i, B_j$ form a $\Z$-basis of $\Z^2$. 
		
		The lemma is proved. 
	\end{proof}
	\begin{lemm}\label{lemm:ABbasis}
		If $A_j, B_j \in \Z^2$ for $j=1,2,3$ and $C\in\Z^2$ satisfy $A_j + B_j =C$ for all $j$, $(\star)$ and $A_i$, $B_j$ form a $\Z$-basis of $\Z^2$, then $A_1 = A_2 = A_3$ and $B_1 = B_2 = B_3$. 
	\end{lemm}
	\begin{proof}
		Since $A_1$ and $B_2$ form a $\Z$-basis of $\Z^2$ and $C \in \cone (A_1,B_2)$, 
		there exists $a, b \in \Z_{>0}$ such that $C = aA_1+bB_2$. 
		Then, it follows from $C = A_1 + B_1 = A_2 + B_2$ that $A_2 = aA_1 + (b-1)B_2$ and $B_1 = (a-1)A_1 + bB_2$. 
		Since $A_2$ and $B_1$ form a $\Z$-basis of $\Z^2$,  we have  
		\begin{equation*}
			\det \begin{pmatrix}
				a & b-1\\
				a-1 & b 
			\end{pmatrix}
			= \pm 1.
		\end{equation*}
		It turns out that $a+b - 1 = \pm 1$. 
		This together with $a,b>0$ yields that $a = b=1$.
		Therefore, $A_2 = A_1$ and $B_1 = B_2$. By the same argument as above, $A_3 = A_1$ and $B_3 = B_2$. This completes the proof of the lemma. 
\end{proof}
\begin{prop}\label{prop:free}
	Let $\rho_L, \rho_R \co (S^1)^2 \to T$ be smooth homomorphisms and $A_j$, $B_j$, $C$ as Theorem \ref{thm:mainthm}. Assume that $A_j, B_j$, and $C$ satisfy the condition $(\star)$. Then, the action of $(S^1)^2$ on $SU(3)$ given by $(\rho_L, \rho_R)$  is free if and only if $\rho_L$ is trivial and $\rho_R$ is an isomorphism. 
\end{prop}
\begin{proof}
	The ``only if" part is obvious. For the ``if" part, assume that the action of $(\rho_L, \rho_R)((S^1)^2)$ on $SU(3)$ is free. By Lemmas \ref{lemm:free} and \ref{lemm:ABbasis},  $A_1 = A_2 = A_3$ and $B_1=B_2=B_3$. Thus, $w_1^L = w_2^L = w_3^L$.  Since $\sum_{j=1}^3 w_j^L=0$, we have $w_j^L=0$ for all $j$. This implies  that $\rho_L$ is trivial, and $w_1^R = -A_1 = -A_2 = -A_3$ and $w_3^R = B_1= B_2 = B_3$. Since $w_1^R$ and  $w_3^R$ form a $\Z$-basis of $\Z^2$,  we have $\rho_R$ is an isomorphism. 
\end{proof}
\begin{rema}\label{remafree}
When $\rho_{L}$ is trivial, 
the set $\Phi^{-1}(C)$ is an orbit of $SU(3)$-action on $SL(3,\C)/U$.
Since the complex structure, the K\"ahler structure, and the fundamental vector fields of the $\{1\}\times T$-action on $SL(3,\C)/U$ are invariant under  $SU(3)$-action on $SL(3,\C)/U$, a complex structure and a transverse K\"ahler structure on $SU(3)$ given by Theorem \ref{thm:mainthm} are left-invariant.

	Proposition \ref{prop:free} yields that, under the condition $(\star)$, if the action of $(\rho_L, \rho_R)((S^1)^2)$ is free, then the quotient $SU(3)/(\rho_L, \rho_R)((S^1)^2)$ is nothing but a flag manifold with an invariant K\"ahler structure. 
\end{rema}
Finally, we see a nontrivial example of $(\rho_L, \rho_R)$ that satisfies $(\star)$. 
\begin{exam}
	We begin with the configuration of points $A_1, A_2, A_3, B_1, B_2, B_3, C \in \R^2$. Since $C \notin \cone (A_i, A_j)$ for all $i, j=1,2,3$, the three points $A_1,A_2$ and $A_3$ should sit in the same half plane whose boundary is the line spanned by $C$. For example, choose $A_1, A_2, A_3$, and $C$ as 
	\begin{equation*}
		A_1= A_2 = (1,0), \quad A_3 = (2,-1), \quad C = (1,1). 
	\end{equation*}
	We choose $B_1, B_2, B_3$ as 
	\begin{equation*}
		B_1 = B_2 =(0,1), \quad B_3 =(-1,2)
	\end{equation*}
	so that $A_j + B_j =C$. Then, $A_1, A_2, A_3, B_1, B_2, B_3$, and $C$ satisfy $(\star)$ and $A_1 + B_1 = A_2 + B_2 = A_3 + B_3$  (see the figure below).
	
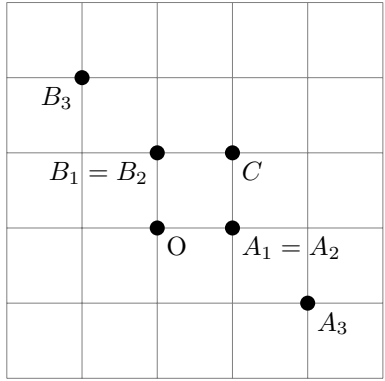
\begin{center}
		\begin{tikzpicture}
			\draw[help lines] (0,0) grid (5,5); 
				\node at (2,2) [below right] {$\mathrm O$};
				\fill (2,2) circle (.1);
				
				\node at (3,2) [below right] {$A_1 = A_2$};
				\fill (3,2) circle (.1);
				\node at (4,1) [below right] {$A_3$};
				\fill (4,1) circle (.1);
				
				\node at (3,3) [below right] {$C$};
				\fill (3,3) circle (.1);
				
				\node at (2,3) [below left] {$B_1 = B_2$};
				\fill (2,3) circle (.1);
				
				\node at (1,4) [below left] {$B_3$};
				\fill (1,4) circle (.1);
		\end{tikzpicture}
	\end{center}
	Now we solve the linear equations
	\begin{equation*}
		A_j = w_j^L - w_1^R, \quad B_j = -w_j^L +w_3^R, \quad \sum_{j=1}^3 w_j^L = \sum_{j=1}^3 w_j^R =0
	\end{equation*}
	for $w_j^L, w_j^R$, $j=1,2,3$. Then, 
	\begin{equation*}
		\begin{split}
			w_1^L &= -(A_1+A_2+A_3)/3 + A_1 = (-1/3, 1/3),\\
			w_2^L & = -(A_1+A_2+A_3)/3 + A_2 = (-1/3, 1/3), \\
			w_3^L & = -(A_1+A_2+A_3)/3 + A_3 = (2/3, -2/3), \\
			w_1^R &= -(A_1+A_2+A_3)/3 = (-4/3, 1/3),\\
			w_2^R & = (A_1+A_2+A_3)/3 - (B_1+B_2+B_3)/3 = (5/3, -5/3), \\
			w_3^R & = (B_1+B_2+B_3)/3 = (-1/3, 4/3). 
		\end{split}
	\end{equation*}
	By multiplying three so that every entry becomes integers, we obtain the homomorphisms 
	\begin{equation*}
		\begin{split}
			\rho_L (t) &=\diag (t^{3w_1^L}, t^{3w_2^L}, t^{3w_3^L}) = \diag (t_1^{-1}t_2, t_1^{-1}t_2, t_1^2t_2^{-2}), \\
			\rho_R (t) &= \diag (t^{3w_1^R}, t^{3w_2^R}, t^{3w_3^R}) = \diag (t_1^{-4}t_2, t_1^5 t_2^{-5}, t_1^{-1}t_2^4)
		\end{split}
	\end{equation*}
	such that the quotient space $SU(3)/(\rho_L, \rho_R)((S^1)^2)$ has a K\"ahler orbifold structure. 
\end{exam}

\section{Basic and Dolbeault cohomologies}\label{sec:coho}
We consider $SU(3)$ equipped with the complex structure $J$ and the transverse K\"ahler holomorphic  foliation $\mathcal F$ determined by  $(\rho_L, \rho_R) \co (S^1)^2 \to T \times T$ under the assumption $(\star)$, as in Theorem \ref{thm:mainthm}. 
Let $H^*_B(SU(3))$ denote the basic cohomology of $SU(3)$ with real coefficients associated with the foliation $\mathcal F$. In order to describe $H^*_B(SU(3))$ we apply \cite[Theorem 4.13]{IK2019}. By \cite[Theorem 4.13]{IK2019}, there exist a real $2$-dimensional vector space $W$ and a differential $d$ on  $H^*_B(SU(3)) \otimes \bigwedge W$ such that 
\begin{itemize}
	\item $dW\subset H^{2}_{B}(SU(3))$;
	\item $d$ on $H^{*}_{B}(SU(3))$ is trivial;
	\item $H^*(H^*_B(SU(3)) \otimes \bigwedge W, d)\cong H^{*}(SU(3),\R)$, 
\end{itemize}
where the degree of all non-zero element in $W$ is $1$. We denote by $A^*$ the differential graded algebra $(H^*_B(SU(3)) \otimes \bigwedge W, d)$ and by $A^k$ the degree $k$ part of $A^*$. 

$SU(3)$ is diffeomorphic to an orientable $S^3$-bundle over $S^5$. Applying the Leray-Hirsch theorem we have 
	\begin{equation*}
		H^k(SU(3),\R) \cong \begin{cases}
			\R, & k=0,3,5,8,\\
			0 & \text{otherwise}. 
		\end{cases}
	\end{equation*}
	It follows from $H^1(SU(3),\R)=0$ that any $d$-closed element in $A^1 = W\oplus H^1_B(SU(3))$ are $d$-exact. This together with $d|_{H^*_B(SU(3))} = 0$ and $d|_{A^0}=0$ yields that $H^1_B(SU(3))=0$ and $d|_W \co W \to H^2_B(SU(3))$ is injective. It follows from $H^2(SU(3),\R) = 0$ that any $d$-closed element in $A^2 = \bigwedge^2 W \oplus (H^1_B(SU(3)) \otimes W) \oplus H^2_B(SU(3))$ is $d$-exact. This together with $d|_{H^*_B(SU(3))} = 0$ yields that $H^2_B(SU(3)) =dW$. 
	Thus,  for  basis vectors $w_{1},w_{2}$ of $W$, we have $H^2_B(SU(3)) = \langle dw_1, dw_2\rangle$. 
Since  the foliation ${\mathcal F}$ is transverse K\"ahler, $H^{6}_{B}(SU(3))=\langle [\omega]^{3}\rangle$, where $\omega$ is a transverse K\"ahler form with respect to $\mathcal{F}$ (see \cite[Proposition 4.8]{IK2019}).
By the hard Lefschetz property (\cite{EKA}),  the linear map $H^2_B(SU(3)) \to H^4_B(SU(3))$ given by $\alpha \mapsto [\omega]\wedge \alpha$ is an isomorphism. Thus $H^{4}_{B}(SU(3))=\langle  [\omega]\wedge dw_{1},  [\omega]\wedge dw_{2}\rangle$. Also, the linear map $H^1_B(SU(3)) \to H^5_B(SU(3))$ given by $\alpha \mapsto [\omega]^2\wedge \alpha$ is an isomorphism. Therefore $H^5_B(SU(3))=0$.
To describe $H^3_B(SU(3))$, we focus on $H^4(SU(3),\R)=0$. It follows from $d|_{H^*_B(SU(3))} =0$, $dW \subset H^2_B(SU(3))$ and $H^5_B(SU(3))=0$ that $d(H^3_B(SU(3))\otimes W) \subset H^5_B(SU(3))=0$. On the other hand, since $A^3 = (H^2_B(SU(3))\otimes W) \oplus H^3_B(SU(3))$, we have $dA^3 \subset H^4_B(SU(3))$. It follows from $H^3_B(SU(3))\otimes W \cap H^4_B(SU(3))=0$ and $H^4(SU(3),\R)=0$ that $H^3_B(SU(3))\otimes W=0$. Since $W \neq 0$, finally we have $H^3_B(SU(3))=0$.

In summary,  
\begin{equation*}
	\dim H^k_B(SU(3)) =\begin{cases}
		1, & k=0, 6,\\
		2, & k=2,4, \\
		0 & \text{otherwise}.
	\end{cases}
\end{equation*}
\begin{rema}
The basic cohomology $H^{*}_{B}(SU(3))$ is canonically isomorphic to the singular cohomology $H^{\ast}(SU(3)/(\rho_L, \rho_R)((S^1)^2),\R)$ of the topological space $SU(3)/(\rho_L, \rho_R)((S^1)^2)$ because the action of $(S^1)^2$ is locally free (see \cite[Corollary 5.3.3]{Pfl2004}).
 \end{rema}

We consider the basic Dolbeault cohomology $H^{*,*}_{B}(SU(3))$ for the holomorphic  foliation $\mathcal F$.
 Then, we have the Hodge decomposition $H^{r}_{B}(SU(3))\otimes\C=\bigoplus_{p+q=r} H^{p,q}_{B}(SU(3))$ (\cite{EKA}).
A transverse K\"ahler form $\omega$ on $SU(3)$ with respect to $\mathcal F$ is a $(1,1)$-form. Therefore $\dim H_B^{1,1}(SU(3))$ is at least $1$. By the symmetry of basic Hodge numbers for the transverse K\"ahler foliation, $H^{2,0}_B(SU(3)) \cong H^{0,2}_B(SU(3))$. Thus the basic Hodge decomposition $H^2_B(SU(3)) \otimes \C \cong H^{2,0}_B(SU(3)) \oplus H^{1,1}_B(SU(3)) \oplus H^{0,2}_B(SU(3))$ and $\dim H^2_B(SU(3))=2$ imply that $H^{2}_{B}(SU(3))\otimes\C=H^{1,1}_{B}(SU(3))$. By the Serre duality of basic Hodge numbers, $H^{4}_{B}(SU(3))\otimes\C=H^{2,2}_{B}(SU(3))$. Thus 
\begin{equation}\label{eq:basicDolbeault}
	\dim H^{p,q}_B(SU(3)) = \begin{cases}
		1, & (p,q) = (0,0), (3,3),\\
		2, & (p,q) = (1,1), (2,2),\\
		0, & \text{otherwise}. 
	\end{cases}
\end{equation}

$H^{*,*}(SU(3))$ denotes the Dolbeault cohomology of $SU(3)$ equipped with  the complex structure determined by $(\rho_L, \rho_R) \co (S^1)^2 \to T \times T$ under the assumption $(\star)$, as in Theorem \ref{thm:mainthm}.
By \cite[Theorem 4.13]{IK2019}, 
there exist a real $2$-dimensional vector space $W$ equipped with a direct sum decomposition $W \otimes \C = W^{1,0} \oplus W^{0,1}$ with $\overline{W^{1,0}} = W^{0,1}$ and a differential $\bar\partial$ on a bi-graded algebra $H^{*,*}_B(SU(3)) \otimes \bigwedge (W^{1,0} \oplus W^{0,1})$ such that 
\begin{itemize}
	\item The degrees of $W^{1,0}$ and $W^{0,1}$ are $(1,0)$ and $(0,1)$, respectively;
	\item $\bar\partial W^{1,0} \subset H^{1,1}_B(SU(3))$ and $\bar \partial W^{0,1} \subset H^{0,2}_B(SU(3))$;
	\item $\bar \partial$ on $H^{*,*}_{B}(SU(3))$ is trivial;
	\item $H^{*,*}(H^{*,*}_B(SU(3)) \otimes \bigwedge (W^{1,0} \oplus W^{0,1}), \bar \partial )\cong H^{*,*}(SU(3))$.
\end{itemize}
We denote by $B^{*,*}$ the differential bi-graded algebra $(H^{*,*}_B(SU(3)) \otimes \bigwedge (W^{1,0} \oplus W^{0,1}), \bar \partial)$. 
	By $H^{*,*}(B^{*,*}) \cong H^{*,*}(SU(3))$ and \eqref{eq:basicDolbeault}, we can compute the Hodge numbers of $SU(3)$ except for $h^{2,1}, h^{2,2}, h^{2,3}$. 
The Hodge diamond is shown below: 
\begin{equation*}
	\begin{tikzpicture}
		\node at (0,0) {$1$};
		\node at (-.5,-.5) {$0$}; 
		\node at (.5,-.5) {$1$};
		\node at (-1,-1) {$0$};
		\node at (0,-1) {$1$};
		\node at (1,-1) {$0$};
		\node at (-1.5, -1.5) {$0$};
		\node at (-.5, -1.5) {$h^{2,1}$};
		\node at (.5, -1.5) {$1$};
		\node at (1.5, -1.5) {$0$};
		\node at (-2, -2) {$0$};
		\node at (-1, -2) {$0$};
		\node at (0, -2) {$h^{2,2}$}; 
		\node at (1,-2) {$0$};
		\node at (2,-2) {$0$};
		\node at (-1.5, -2.5) {$0$};
		\node at (-.5, -2.5) {$1$}; 
		\node at (.5, -2.5) {$h^{2,3}$};
		\node at (1.5, -2.5) {$0$};
		\node at (-1,-3) {$0$};
		\node at (0,-3) {$1$}; 
		\node at (1,-3) {$0$};
		\node at (-.5,-3.5) {$1$};
		\node at (.5,-3.5) {$0$};
		\node at (0, -4) {$1$};
	\end{tikzpicture}
\end{equation*}
The Hodge numbers $( h^{2,1}$, $h^{2,2}$, $h^{2,3})$ are $(0,0,0)$ or  $(1,2,1)$, depending on the complex structure on $SU(3)$. See Remark \ref{rema:W} for detail. 
\begin{rema}
If $\rho_{L}$ is non-trivial, then  the complex structure on $SU(3)$ is not left-invariant.
The  computations of Dolbeault cohomology on compact Lie groups  equipped with left-invariant complex structures (e.g., \cite{Pittie1988})  can not be applied.
\end{rema}

\begin{rema}
We can describe $W$ and  $d\co W\to H^{2}_{B}(SU(3))$ explicitly  for the differential graded algebra  $(A^*=H^{*}_{B}(SU(3))\otimes \bigwedge W, d)$; see \cite[Page 68]{IK2019}.
We use $(-,-)$ to denote the Cartan-Killing form on $SU(3)$. 
For a vector field $X$ on $SU(3)$, we define the $1$-form  $w_{X}=(X,-)$. 
Let $\bar{\frak t}$ be the vector space of fundamental vector fields on $SU(3)$ defined by  the $(S^{1})^{2}$-action associated with $(\rho_{L}, \rho_{R})\co (S^{1})^{2}\to T \times T$.
Then, for any $X\in \bar{\frak t}$, the $1$-form  $w_{X}$ is $T\times T$-invariant.
We can take $W=\{w_{X}\mid  X\in \bar{\frak t}\}$ and  $d\co W\ni w\mapsto [dw]\in H^{2}_{B}(SU(3))$. 
\end{rema}
\begin{rema}\label{rema:W}
For the differential bi-graded algebra $(B^{*,*}=H^{*,*}_{B}(SU(3))\otimes \bigwedge (W^{1,0}\oplus \overline{W^{1,0}}), \bar\partial)$, we have to reconstruct  $W$ so that $W$ is $J$-invariant see \cite[Proposition 4.9]{IK2019}.
We have a direct sum $TN=T{\mathcal F}\oplus T{\mathcal F}^\perp$ so that $T{\mathcal F}^\perp$  is $J$-invariant and invariant under  the $(S^{1})^{2}$-action associated with $(\rho_{L}, \rho_{R})\co (S^{1})^{2}\to T\times T$, see Section \ref{sec:transverseKaehler}.
Regarding $TN^*=T{\mathcal F}^*\oplus (T{\mathcal F}^\perp)^*$,
we should replace $W$ with the vector space of $1$-forms generated by  the locally free  $(S^{1})^{2}$-action associated with $\rho_{L}, \rho_{R}: (S^{1})^{2}\to T$. 
A direct sum decomposition $W\otimes \C=W^{1,0}\oplus \overline{W^{1,0}}$ corresponds to the parameter in $GL(2,\R)/GL(1,\C)$ as in Remark \ref{rem-cstr2}.
\end{rema}

\end{document}